\documentclass[a4paper,12pt]{article}     

\usepackage{graphicx}		  
\usepackage{amsmath,amssymb}	

\newtheorem{lemma}{Lemma}
\newtheorem{theorem}[lemma]{Theorem}

\newtheorem{corollary}[lemma]{Corollary}

\newenvironment{proof}{
\hspace*{-9mm}
{ \it Proof.}}
{\hfill {$\square$}\vspace{1.5em}}

\begin{document}

\begin{center}{{\Large Separating a chart }
\vspace{10pt}
\\ 
Teruo NAGASE and Akiko SHIMA\footnote
{
The second author is partially supported by Grant-in-Aid for Scientific Research (No.23540107), Ministry of Education, Science and Culture, Japan. 
}
}
\end{center}

\begin{abstract}
In this paper, we shall show a condition for that a chart
is C-move equivalent to the product of two charts, the union of two charts $\Gamma^*$ and $\Gamma^{**}$ which are contained in disks $D^*$ and $D^{**}$ with $D^*\cap D^{**}=\emptyset$.
\end{abstract}

\setcounter{section}{0}
\section{{\large Introduction}}


Charts are oriented labeled graphs in a disk which represent surface braids (see  \cite{KnottedSurfaces},\cite{BraidBook}, and see Section~\ref{s:Prel} for the precise definition of charts, see \cite[Chapter 14]{BraidBook} for the definition of surface braids).
In a chart there are three kinds of vertices; white vertices, crossings and black vertices. 
A C-move is a local modification 
of charts in a disk.
The closures of surface braids are embedded closed oriented surfaces in 4-space ${\Bbb R}^4$
 (see \cite[Chapter 23]{BraidBook} for the definition of the closures of surface braids). 
A C-move between two charts induces an ambient isotopy between the closures of the corresponding two surface braids.
In this paper, we shall show a condition for that a chart
is C-move equivalent to the product of two charts (Theorem~\ref{MainTheorem}).

We will work in the PL category or smooth category. All submanifolds are assumed to be locally flat.
In \cite{INS} and \cite{NST},
we investigated minimal charts with exactly four white vertices.
In \cite{ONS},
we showed that there is no minimal chart with exactly five vertices. 
Hasegawa proved that there exists a minimal chart with exactly
six white vertices which represents the surface braid whose
closure is ambient isotopic to a 2-twist spun trefoil \cite{H1}.
In \cite{Gambit} and \cite{NNS},
we investigated minimal charts with exactly six white vertices.
We show that there is no minimal chart with exactly seven vertices
(\cite{ChartAppl}, \cite{ChartAppII}, \cite{ChartAppIII}, \cite{ChartAppIV}, \cite{ChartAppV}). 
Thus  the next targets are minimal charts with eight or nine white vertices.
As an application of Theorem~\ref{MainTheorem},
we shall show that there are 12 kinds of types for a minimal chart with eight white vertices
(Corollary~\ref{MinimalChart8Zero} and Corollary~\ref{MinimalChart8}), 
and there are 15 kinds of types for a minimal chart with nine white vertices
(Corollary~\ref{MinimalChart9}).

Let $\Gamma$ be a chart. For each label $m$, we denote by $\Gamma_m$ the 'subgraph' of $\Gamma$ consisting of all the edges of label $m$ and their vertices.

Two charts are said to be {\it C-move equivalent} 
if there exists
a finite sequence of C-moves 
which modifies one of the two charts 
to the other.

Now we define a type of a chart:
Let $\Gamma$ be a chart, $m$ a label of $\Gamma$, and $n_1,n_2,\dots,n_p$ integers.
The chart $\Gamma$ is said to be of {\it type $(m;n_1,n_2,\dots,n_p)$}, 
or of {\it type $(n_1,n_2,\dots,n_p)$}
briefly,
if it satisfies the following three conditions:
\begin{enumerate}
\item[(i)] 
For each $i=1,2,\dots, p$, 
the chart $\Gamma$ contains 
exactly $n_{i}$ white vertices 
in $\Gamma_{m+i-1}\cap \Gamma_{m+i}$.
\item[(ii)] 
If $i<0$ or $i>p$, 
then $\Gamma_{m+i}$ does not contain any white vertices.
\item[(iii)] 
Each of the two subgraphs $\Gamma_m$ and $\Gamma_{m+p}$ contains at least one white vertex.
\end{enumerate}
Note that $n_1\ge1$ and $n_p\ge1$ by Condition (iii).

Two C-move equivalent charts $\Gamma$ and $\Gamma'$ are said to be {\it same C-type} provided that
the types of the two charts are same.

For a subset $X$ of a chart, let
\begin{enumerate}
\item[] 
$w(X)=$ the number of white vertices of the chart contained in $X$. 
\end{enumerate}

A chart $\Gamma$ is {\it zero at label $k$} 
provided that
\begin{enumerate}
\item[(i)] 
$\Gamma_k\cap\Gamma_{k+1}=\emptyset$,
\item[(ii)] 
there exists a label $i$ with 
$i\leq k$ and $w(\Gamma_i)\neq 0$, and
\item[(iii)] 
there exists a label $j$ with 
$k<j$ and $w(\Gamma_j)\neq 0$.
\end{enumerate}

Let $\Gamma$ be an $n$-chart, and 
$D_1,D_2$ disjoint disks 
with 
$D_i\cap\Gamma\neq\emptyset$ for $i=1,2$,
$\partial D_i\cap\Gamma=\emptyset$ for $i=1,2$, 
and
$D_1\cup D_2\supset\Gamma$.
Then we can consider 
$\Gamma^*=D_1\cap\Gamma$ and
$\Gamma^{**}=D_2\cap\Gamma$ as $n$-charts.
Then we call the chart $\Gamma$ the 
{\it product of the two charts 
$\Gamma^*$ and $\Gamma^{**}$}.

A chart $\Gamma$ is {\it separable at label $k$} 
if there exist subcharts $\Gamma^*$, $\Gamma^{**}$ such that
\begin{enumerate}
\item[(i)] $\Gamma$ is the product of the two charts $\Gamma^*$ and $\Gamma^{**}$,
\item[(ii)] $w(\Gamma^{*})\neq 0$ and
$w(\Gamma^{**})\neq 0$,
\item[(iii)] 
$w(\Gamma_i^{*})=0$ 
for all label $i$ with 
$k<i$, and
\item[(iv)] 
$w(\Gamma_i^{**})=0$ 
for all label $i$ with 
$i\leq k$.
\end{enumerate}

The following is our main theorem:
\begin{theorem}
\label{MainTheorem}
A chart $\Gamma$ is zero at label $k$ 
if and only if there exists a chart $\Gamma'$
with the same C-type of $\Gamma$ such that 
$\Gamma'$ is separable at label $k$.
\end{theorem}

A chart $\Gamma$ is 
{\it minimal} if
it possesses the smallest number of white vertices
among the charts 
C-move equivalent to 
the chart $\Gamma$ (cf. \cite{BraidThree}).

\begin{corollary}
\label{MinimalChart8Zero}
Let $\Gamma$ be a minimal chart with $w(\Gamma)=8$.
If $\Gamma$ is zero at some label,
then $\Gamma$ is C-move equivalent to one of the following charts:
\begin{enumerate}
\item[{\rm (a)}] the product of two charts of type $(4)$.
\item[{\rm (b)}] the product of two charts of type $(2,2)$.
\item[{\rm (c)}] the product of a chart of type $(4)$ and a chart of type $(2,2)$.
\end{enumerate}
\end{corollary}

\begin{corollary}
\label{MinimalChart8}
Let $\Gamma$ be a minimal $n$-chart with $w(\Gamma)=8$ such that
$\Gamma$ is not zero at any label.
If necessary
we change all the edges of label $k$ to 
ones of label $n-k$
for each $k=1,2,\cdots,n-1$
simultaneously,
then the type of $\Gamma$ is
$(8)$, $(6,2)$, $(5,3)$, $(4,4)$, 
$(4,2,2)$, $(3,3,2)$, $(3,2,3)$, $(2,4,2)$ 
or $(2,2,2,2)$.
\end{corollary}

\begin{corollary}
\label{Corollary1}
If $\Gamma$ is a minimal chart 
with $w(\Gamma)=9$ or $11$,
then the chart is not zero at any label. 
Namely 
if the type of the chart is $(m;n_1,n_2,\cdots ,n_p)$, 
then for each $i=1,2,\cdots ,p$, we have $n_i\neq 0$.
\end{corollary}

\begin{corollary}
\label{MinimalChart9}
Let $\Gamma$ be a minimal $n$-chart with $w(\Gamma)=9$.
If necessary
we change all the edges of label $k$ to 
ones of label $n-k$
for each $k=1,2,\cdots,n-1$
simultaneously,
then the type of $\Gamma$ is
$(9)$, $(7,2)$, $(6,3)$, $(5,4)$,
 $(5,2,2)$, $(4,3,2)$, $(4,2,3)$, $(4,1,4)$, 
 $(3,4,2)$, $(3,3,3)$, $(2,5,2)$, 
 $(4,1,2,2)$, $(3,2,2,2)$, $(2,3,2,2)$ 
 or $(2,2,1,2,2)$.
\end{corollary}

The paper is organized as follows.
In Section~\ref{s:Prel},
we define charts.
In Section~\ref{s:SeparatingSystems} and Section~\ref{s:MovableDisk},
we show lemmata to need a proof of Theorem~\ref{MainTheorem}.
In Section~\ref{s:OmegaMinimal},
we define $\omega_k$-minimal charts.
By using $\omega_k$-minimal charts,
we show that if a chart $\Gamma$ is zero at label $k$, 
then there exists a chart $\Gamma^*$ obtained from $\Gamma$ by C-I-R2 moves and C-I-M2 moves
such that $\Gamma_k^*\supset \Gamma_k$
and all of black vertices and white vertices in $\bigcup^{\infty}_{i=k+1}\Gamma_i^*$ are contained in the same complementary domain of $\Gamma_k$.
In Section~\ref{s:ProofMainTheorem},
we give a proof of Theorem~\ref{MainTheorem}
and proofs of corollaries.

\section{{\large Preliminaries}}
\label{s:Prel}

Let $n$ be a positive integer. An {\it $n$-chart} is an oriented labeled graph in a disk,
which may be empty or have closed edges without vertices, called {\it hoops},
satisfying the following four conditions:
\begin{enumerate}
	\item[(i)] Every vertex has degree $1$, $4$, or $6$.
	\item[(ii)] The labels of edges are in $\{1,2,\dots,n-1\}$.
	\item[(iii)] In a small neighborhood of each vertex of degree $6$,
	there are six short arcs, three consecutive arcs are oriented inward and the other three are outward, and these six are labeled $i$ and $i+1$ alternately for some $i$,
	where the orientation and the label of each arc are inherited from the edge containing the arc.
	\item[(iv)] For each vertex of degree $4$, diagonal edges have the same label and are oriented coherently, and the labels $i$ and $j$ of the diagonals satisfy $|i-j|>1$.
\end{enumerate}
We call a vertex of degree $1$ a {\it black vertex,} a vertex of degree $4$ a {\it crossing}, and a vertex of degree $6$ a {\it white vertex} respectively (see Fig.~\ref{Vertices}).

\begin{figure}[t]
\centerline{\includegraphics{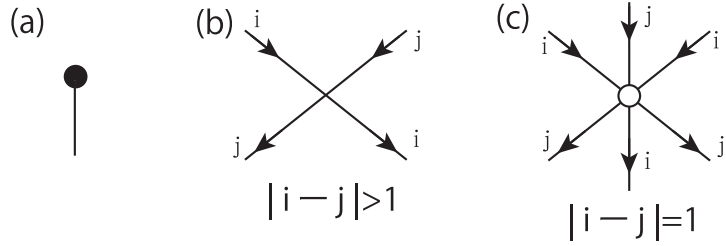}}
\vspace{5mm}
\caption{\label{Vertices}(a) A black vertex. (b) A crossing. (c) A white vertex.}
\end{figure}

Now {\it $C$-moves} are local modifications 
of charts in a disk
as shown in 
Fig.~\ref{Cmove}
(see \cite{KnottedSurfaces}, 
\cite{BraidBook}, \cite{Tanaka} 
for the precise definition).
These C-moves as shown in 
Fig.~\ref{Cmove} 
are examples of C-moves.
We often use
C-I-M2 moves, and C-I-R2 moves in this paper.


\begin{figure}[t]
\centerline{\includegraphics{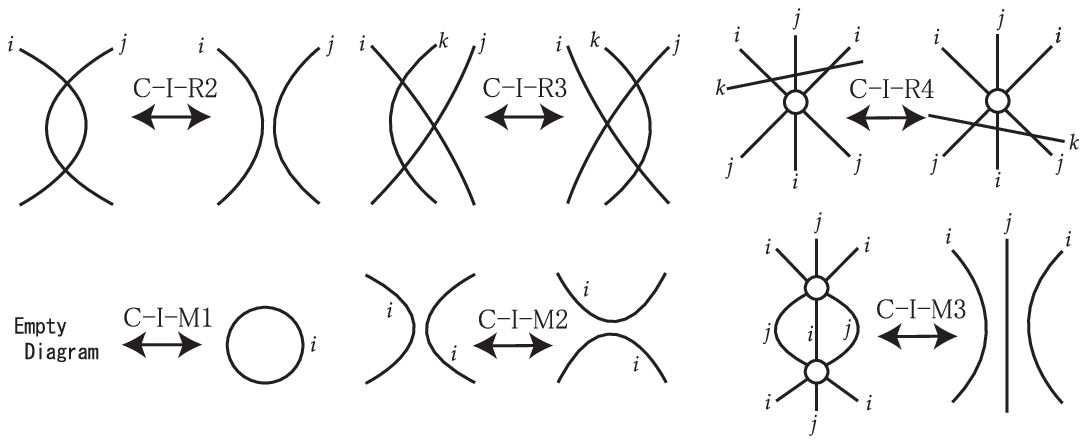}}
\vspace{5mm}
\caption{\label{Cmove}}
\end{figure}

Let $\Gamma$ be a chart,
and $m$ a label of $\Gamma$.
The 'subgraph' $\Gamma_m$ of $\Gamma$ consists of all the edges of label $m$ and their vertices.
An edge of $\Gamma$ is the closure of a connected component of the set obtained by taking out all white vertices and crossings from $\Gamma$.
On the other hand, we assume that 
\begin{enumerate}
	\item[] an {\it edge} of $\Gamma_m$ is the closure of a connected component of the set obtained by taking out all white vertices from $\Gamma_m$.
\end{enumerate}
Thus 
any vertex of $\Gamma_m$ is a black vertex or a white vertex.
Hence any crossing of $\Gamma$ is not considered as a vertex of $\Gamma_m$.

In this paper
for a set $X$ in a space
we denote 
the interior of $X$,
the boundary of $X$ and
the closure of $X$
by Int$X$, $\partial X$
and $Cl(X)$
respectively.


\section{Separating systems}
\label{s:SeparatingSystems}

In this paper we assume that 
{\it every chart is in the plane}.

A disk $D$ is {\it in general position} 
with respect to a chart $\Gamma$ 
provided that
\begin{enumerate}
\item[(i)] $\partial D$ does not contain any white vertices, black vertices, nor crossings of $\Gamma$, and
\item[(ii)] if an edge of $\Gamma$
intersects $\partial D$,
then the edge intersects $\partial D$ transversely.
\end{enumerate}

Let $D$ be a disk.
The simple arc $\ell$ is called a {\it proper arc} of $D$
provided that $\ell\cap\partial D=\partial\ell$.
Let $L$ be a simple arc on $\partial D$. 
A proper arc $\ell$ of $D$ is called a {\it $(D,L)$-arc} 
provided that $\partial \ell\subset L$.

Let $\Gamma$ be a chart, and
$k$ a label of $\Gamma$.
A simple arc in an edge of $\Gamma_k$
is called an {\it arc of label $k$}.

Let $D$ be a disk in general position with respect to a chart $\Gamma$, and
$L$ a simple arc on $\partial D$.
A $(D,L)$-arc $\ell$ of label $k$ is said to be {\it reducible}
if for the subarc $L'$ of $L$ with $\partial L'=\partial \ell$ 
we have 
Int$L'\cap (\Gamma_{k-1}\cup\Gamma_{k}\cup\Gamma_{k+1})=\emptyset$.

For a subset $X$ of a chart $\Gamma$, let
\begin{center}
${\Bbb {BW}}(X)=$ the set of all the black vertices and white vertices of $\Gamma$ in $X$.
\end{center}

Let $k$ be a positive integer.
Let $D^-$ and $D^+$ be disks 
in general position 
with respect to a chart $\Gamma$ 
such that 
$J=D^-\cap D^+$ is an arc. 
The triplet $(D^-,D^+,J)$ is called a 
{\it separating system at label $k$} 
for the chart $\Gamma$ 
provided that
\begin{enumerate}
\item[(i)] 
$D^-\cap {\Bbb {BW}}(\bigcup_{i=k+1}^\infty\Gamma_i)=\emptyset$,
\item[(ii)] 
$D^+\cap {\Bbb {BW}}(\bigcup_{i=1}^k\Gamma_i)=\emptyset$,
\item[(iii)] 
$J\cap \Gamma_k=\emptyset$, and
\item[(iv)] 
$Cl(\partial D^+-J)\cap \Gamma=\emptyset$.
\end{enumerate}

\begin{lemma}
\label{LemmaReducible(D^+,J)-arc}
Let $(D^-,D^+,J)$ be a separating system at label $k$ 
for a chart $\Gamma$.
Then the following hold:
\begin{enumerate}
\item[{\rm (a)}] If there exists 
a $(D^+,J)$-arc of label less than $k$, 
then there exists a 
reducible $(D^+,J)$-arc of label less than $k$.
\item[{\rm (b)}] If there exists 
a $(D^-,J)$-arc of label greater than $k$, 
then there exists a 
reducible $(D^-,J)$-arc of label greater than $k$.
\end{enumerate}
\end{lemma}

\begin{proof} 
We show Statement (a).
Suppose that there exists a 
$(D^+,J)$-arc of label less than $k$.
Let $\Bbb S$ be the set of 
all $(D^+,J)$-arcs of label less than $k$.
For each $\ell\in \Bbb S$ let $P_\ell$ be 
the subarc of $J$ 
with $\partial P_\ell=\partial \ell$. 
Let
\begin{enumerate}
\item[] 
$m(\ell)=$ the number of points in 
${\rm Int}(P_\ell)\cap (\bigcup_{i=1}^k\Gamma_i).$
\end{enumerate}
Let $\ell_0$ be an element in $\Bbb S$ with
\begin{enumerate}
\item[] 
$m(\ell_0)=\min\{~m(\ell)~|~\ell\in\Bbb S~\}.$
\end{enumerate}
Let $j$ be the label of $\ell_0$.
We have $j\leq k-1$.
Let
\begin{enumerate}
\item[] $s=$ the number of points 
in ${\rm Int}(P_{\ell_0})\cap 
(\Gamma_{j-1}\cup\Gamma_{j}\cup\Gamma_{j+1})$.
\end{enumerate}
We show that $s=0$ by contradiction.
Suppose that $s>0$.
Let $D_0$ be the disk in $D^+$ bounded by $\ell_0\cup P_{\ell_0}$.
Let $\ell_1$ be a connected component of 
$D_0\cap 
(\Gamma_{j-1}\cup\Gamma_{j}\cup\Gamma_{j+1})$ 
different from $\ell_0$. 
Since $j+1\leq (k-1)+1=k$ and 
since there do not exist 
any black vertices nor white vertices of 
$\bigcup_{i=1}^k\Gamma_i$ in $D^+$
by Condition (ii) of a separating sysytem, 
the arc $\ell_1$ is a proper arc of $D_0$.
Since $\ell_0\cap \ell_1=\emptyset$, 
the arc $\ell_1$ is a $(D_0,P_{\ell_0})$-arc.
Hence $\ell_1$ is a $(D^+,J)$-arc.
Let $p$ be the label of $\ell_1$.
Now $J\cap \Gamma_k=\emptyset$ implies 
$p\neq k$. 
Hence we have $p<k$.
Thus $\ell_1\in {\Bbb S}$.
Since $P_{\ell_1}\cap \ell_0=\emptyset$, 
we have
\begin{enumerate}
\item[] 
$m(\ell_1)\leq m(\ell_0)-2$.
\end{enumerate}
This contradicts that $m(\ell_0)$ is minimal.
Hence we have $s=0$.
Therefore
\begin{enumerate}
\item[] 
${\rm Int}(P_{\ell_0})\cap 
(\Gamma_{j-1}\cup\Gamma_{j}\cup\Gamma_{j+1})=\emptyset$.
\end{enumerate}
This means that $\ell_0$ is a desired reducible $(D^+,J)$-arc.
This proves Statement (a).

Similarly we can show Statement (b).
\end{proof}

Let $\Gamma$ be a chart, 
and $e_1$ and $e_2$ edges of label $m$
(possibly $e_1=e_2$).
Let $\alpha$ be an arc such that
\begin{enumerate}
\item[(i)]
$\partial \alpha$ consists of a point in $e_1$
and a point in $e_2$,
and
\item[(ii)]
Int$(\alpha)$ transversely intersects edges of $\Gamma$
(see Fig.~\ref{FigureSurgery}(a)).
\end{enumerate}
Let $D$ be a regular neighborhood of
the arc $\alpha$.
Let $\gamma_1=e_1\cap D$ and
$\gamma_2=e_2\cap D$.
Then $\gamma_1$ and $\gamma_2$ are
proper arcs of $D$ and
they split the disk $D$ into three disks.
Let $E$ be the one of the three disks
with $E\supset \alpha$
(see Fig.~\ref{FigureSurgery}(b)). 
A chart {\it $\Gamma^\prime$ 
is obtained from $\Gamma$
by a surgery along $\alpha$}
provided that
\begin{enumerate}
\item[(iii)]
$\Gamma^\prime_m=
(\Gamma_m-(\gamma_1\cup \gamma_2))
\cup
Cl(\partial E-(\gamma_1\cup\gamma_2))$,
and
\item[(iv)]
$\Gamma'_i=\Gamma_i$ $(i\neq m)$
(see Fig.~\ref{FigureSurgery}(c)).
\end{enumerate}
\begin{figure}[h]
\centerline{\includegraphics{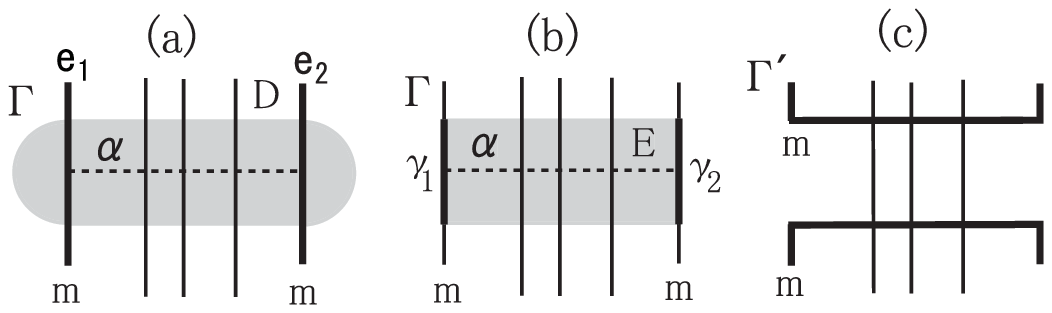}}
\caption{\label{FigureSurgery}}
\end{figure}


Let $k$ be a positive integer.
Let $\Gamma$ and $\Gamma^*$ be charts.
We write $\Gamma^*\overset{k}{\sim}\Gamma$
provided that
\begin{enumerate}
\item[(i)] 
the chart $\Gamma^*$ is obtained from $\Gamma$ 
by applying C-I-M2 moves, C-I-R2 moves and
ambient isotopies of the plane, and
\item[(ii)] 
$\Gamma_k$ is a subset of $\Gamma_k^*$.
\end{enumerate}
For a positive integer $k$ and a chart $\Gamma$,  
let
\begin{enumerate}
\item[] $Fix(\Gamma_k;\Gamma)=
\{~\Gamma^*~|~\Gamma^*\overset{k}{\sim}\Gamma~\}$.
\end{enumerate}
{\bf Remark.} The relation $\Gamma^*\overset{k}{\sim}\Gamma$
implies that $\Gamma^*$ and $\Gamma$ are same C-type.


\begin{lemma} 
\label{Lemma(D,J)-arcFree}
Let $(D^-,D^+,J)$ be a separating system 
at label $k$ 
for a chart $\Gamma$.
Then there exists a chart 
$\Gamma'\in Fix(\Gamma_k;\Gamma)$ such that
\begin{enumerate}
\item[{\rm(a)}] 
the chart $\Gamma'$ is obtained from $\Gamma$ 
by applying surgeries along subarcs of $J$, and
\item[{\rm (b)}] 
the chart $\Gamma'$ does not possess any 
$(D^-,J)$-arcs of label greater than $k$ nor
$(D^+,J)$-arcs of label less than $k$.
\end{enumerate}
\end{lemma}

\begin{proof} 
Let $\Bbb S$ be the set of all charts obtained from $\Gamma$ 
by applying surgeries along subarcs of $J$.
Since $J\cap \Gamma_k=\emptyset$
by Condition (ii) of a separating system,
we have that $\Gamma^*_k=\Gamma_k$ for each chart $\Gamma^*\in{\Bbb S}$.
Thus for each chart $\Gamma^*\in \Bbb S$ 
we have that
$\Gamma^*\in Fix(\Gamma_k;\Gamma)$
and that 
$(D^-,D^+,J)$ is a separating system at label $k$ for $\Gamma^*$.
For each chart $\Gamma^*\in \Bbb S$, let
\begin{enumerate}
\item[] 
$
\begin{array}{lcl}
n(\Gamma^*)&=&
\text{the number of }(D^+,J)\text{-arcs of label less than }k\\
&+&
\text{the number of }(D^-,J)\text{-arcs of label greater than }k.
\end{array}
$
\end{enumerate}
Let $\Gamma'$ be a chart in $\Bbb S$ such that 
\begin{enumerate}
\item[] $n(\Gamma')=\min\{~n(\Gamma^*)~|~\Gamma^*\in \Bbb S~\}.$
\end{enumerate}
We show that $n(\Gamma')=0$ by contradiction.
Suppose that $n(\Gamma')>0$.
Then by Lemma~\ref{LemmaReducible(D^+,J)-arc} 
the chart $\Gamma'$ possesses a reducible $(D^-,J)$-arc 
of label greater than $k$ or 
a reducible $(D^+,J)$-arc
of label less than $k$, say $\ell$.
Let $P_{\ell}$ be the subarc of $J$ with $\partial P_{\ell}=\partial \ell$. 
Let $\Gamma''$ be a chart obtained from $\Gamma'$ 
by applying a surgery along the subarc $P_{\ell}$.
Then we have $\Gamma''\in\Bbb S$ and
\begin{enumerate}
\item[] $n(\Gamma'')\leq n(\Gamma')-1$.
\end{enumerate}
This contradicts that $n(\Gamma')$ is minimal.
Therefore $n(\Gamma')=0$.
The chart $\Gamma'$ is a desired chart.
\end{proof}

\section{Movable disks}
\label{s:MovableDisk}

Let $D$ be a disk in general position 
with respect to a chart $\Gamma$.
The disk $D$ is called a {\it movable disk} 
at label $k$ 
with respect to the chart $\Gamma$ 
provided that
\begin{enumerate}
\item[(i)] 
$D\cap {\Bbb{BW}} (\bigcup_{i=1}^k\Gamma_i)=\emptyset$, and
\item[(ii)]
$\partial D\cap (\bigcup_{i=1}^{k+1}\Gamma_i)=\emptyset$.
\end{enumerate}

Let $\Gamma$ be a chart,
and $m$ a label of $\Gamma$.
A {\it hoop} is a closed edge of $\Gamma$ without vertices 
(hence without crossings, neither).
A {\it ring} is a closed edge of $\Gamma_m$ containing a crossing but not containing any white vertices.


\begin{lemma}
\label{LemmaMovableDisk}
Let $(D^-,D^+,J)$ be a separating system at label $k$ 
for  a chart $\Gamma$.
If $\partial (D^-\cup D^+)\cap \Gamma_{k+1}=\emptyset$, 
then there exists a chart 
$\Gamma'\in Fix(\Gamma_k;\Gamma)$ 
such that
\begin{enumerate}
\item[{\rm (a)}] 
the chart $\Gamma'$ is obtained from $\Gamma$ 
by applying surgeries along subarcs of $J$, and
\item[{\rm (b)}] 
$D^+$ is a movable disk at label $k$ 
with respect to $\Gamma'$.
\end{enumerate}
\end{lemma}

\begin{proof}
By Lemma~\ref{Lemma(D,J)-arcFree} 
there exists a chart 
$\Gamma'\in Fix(\Gamma_k;\Gamma)$ 
obtained from $\Gamma$ 
by applying surgeries along subarcs of $J$ 
such that 
\begin{enumerate}
\item[(1)] $\Gamma'$ does not possess any 
$(D^-,J)$-arcs of label greater than $k$ 
nor 
$(D^+,J)$-arcs of label less than $k$.
\end{enumerate}
Thus by Condition (ii), (iii) and (iv) of 
a separating system, 
\begin{enumerate}
\item[(2)] $D^+\cap {\Bbb{BW}}(\bigcup_{i=1}^k\Gamma'_i)=
D^+\cap {\Bbb{BW}}(\bigcup_{i=1}^k\Gamma_i)=\emptyset$,
\item[(3)] $J\cap\Gamma'_k=J\cap\Gamma_k=\emptyset$,
\item[(4)] $Cl(\partial D^+-J)\cap\Gamma'=Cl(\partial D^+-J)\cap\Gamma=\emptyset$.
\end{enumerate}
Hence
any connected component of 
$D^+\cap (\bigcup_{i=1}^k\Gamma'_i)$ 
is a hoop, a ring or a $(D^+,J)$-arc of label less than $k$.
Since $\Gamma'$ does not possess 
any $(D^+,J)$-arcs of label less than $k$
by (1), 
we have
\begin{enumerate}
\item[(5)] 
$J\cap (\bigcup_{i=1}^k\Gamma'_i)=\emptyset$.
\end{enumerate}
Since the chart $\Gamma'$ is obtained from $\Gamma$ 
by applying surgeries along subarcs of $J$,
we have
\begin{enumerate}
\item[(6)] 
$\partial (D^-\cup D^+)\cap \Gamma'_{k+1}=
\partial (D^-\cup D^+)\cap \Gamma_{k+1}=\emptyset$, and
\item[(7)] 
$D^-\cap{\Bbb{BW}}(\bigcup_{i=k+1}^\infty\Gamma'_i)=
D^-\cap{\Bbb{BW}}(\bigcup_{i=k+1}^\infty\Gamma_i)=\emptyset$ by Condition (i) of a separating system.
\end{enumerate}
Hence
any connected component of $D^-\cap \Gamma'_{k+1}$ 
is a hoop, a ring or a $(D^-,J)$-arc.
Since $\Gamma'$ does not possess any 
$(D^-,J)$-arcs of label greater than $k$
by (1), we have $J\cap \Gamma'_{k+1}=\emptyset$.
Thus (4) and (5) 
imply
\begin{enumerate}
\item[] 
$\partial D^+\cap(\bigcup_{i=1}^{k+1}\Gamma'_i)=\emptyset$.
\end{enumerate}
Therefore $D^+$ is a movable disk at label $k$ 
with respect to $\Gamma'$.
\end{proof}

Let $E$ be a disk 
in general position 
with respect to a chart $\Gamma$.
The disk $E$ is called a 
{\it c-disk at label $k$} 
with respect to $\Gamma$
provided that
in Int$E$ 
there exist mutually disjoint 
movable disks 
$D_1,D_2,\cdots ,D_s$ 
at label $k$ 
with respect to the chart $\Gamma$ 
and 
a connected component $W$ 
of 
$E-(\Gamma_k-\bigcup_{i=1}^sD_i)$ 
(see Fig.~\ref{FigureCdisk}) 
such that
\begin{enumerate}
\item[(i)] $W\supset\partial E$,
\item[(ii)] 
$W\supset\bigcup_{i=1}^s D_i$, and
\item[(iii)] 
$W\supset E\cap {\Bbb{BW}}(\bigcup_{i=k+1}^\infty\Gamma_i)$.
\end{enumerate}
We call $W$ the {\it principal domain} of the c-disk $E$.
We also call the movable disks $D_1,D_2,\cdots ,D_s$ 
{\it associated movable disks} of the c-disk $E$.

\begin{figure}[h]
\centerline{\includegraphics{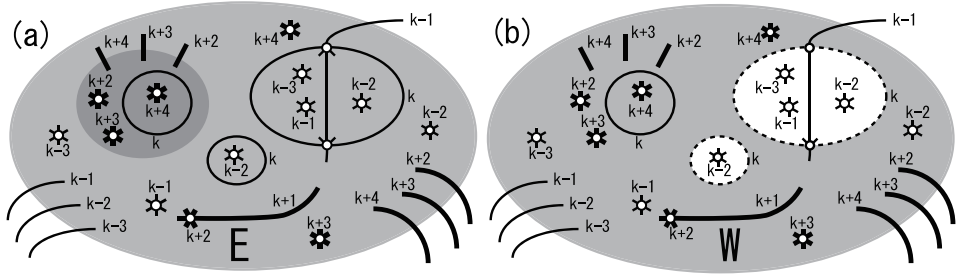}}
\caption{\label{FigureCdisk}(a) The dark disk is a movable disk in a c-disk. (b) The gray 'disk with two holes' is the principal domain of the c-disk.}
\end{figure}

\begin{lemma}
\label{LemmaFromC-diskToSeparatingSystem}
Let $E$ be a c-disk at label $k$ 
with respect to a chart $\Gamma$.
Let $p\in \partial E-\Gamma$. 
Then there exists a separating system 
$(D^-,D^+,J)$ at label $k$
for the chart $\Gamma$ with
$E=D^-\cup D^+$ and $p\in \partial D^+-J$.
\end{lemma}

\begin{proof}
Let $W$ be 
the principal domain of the c-disk $E$ 
and $D_1,D_2,\cdots ,D_s$ 
associated movable disks of $E$.
Suppose that
\begin{enumerate}
\item[]
$(W-\bigcup_{i=1}^sD_i)\cap {\Bbb{BW}}(\bigcup_{i=k+1}^\infty\Gamma_i)
=\{~w_{s+1},w_{s+2},\dots ,w_{t}~\}$.
\end{enumerate}
For each $i=s+1,s+2,\dots, t$, let
$D_i$ be a regular neighbourhood of $w_i$ in $W$.
Then by Condition (iii) of a c-disk
we have
\begin{enumerate}
\item[(1)] 
$E\cap {\Bbb{BW}}(\bigcup_{i=k+1}^\infty\Gamma_i)
\subset \bigcup_{i=1}^tD_i$.
\end{enumerate}
Since $W$ is a connected component of 
$E-(\Gamma_k-\bigcup_{i=1}^sD_i)$ 
by the condition of a c-disk,
we have
\begin{enumerate}
\item[] 
$(W-\bigcup_{i=1}^sD_i)\cap\Gamma_k=\emptyset$.
\end{enumerate}
Hence there exist mutually disjoint 
simple arcs 
$\ell_1,\ell_2,\cdots ,\ell_t$ in $W$ 
(see Fig.~\ref{FigureFromCdiskToSeparatingSystem}(a)) 
such that 
\begin{enumerate}
\item[(2)] 
for each $i=1,2,\cdots ,t,$ 
$\ell_i\cap\Gamma_k=\emptyset$, and 
$\partial \ell_i\cap\Gamma=\emptyset$,
\item[(3)] 
for each $i=1,2,\cdots ,t,$ 
the arc $\ell_i$ does not contain any white vertices, black vertices nor crossings of $\Gamma$, and 
the arc $\ell_i$ intersects edges of $\Gamma$ transversely,
\item[(4)] 
$(\bigcup_{i=1}^t {\rm Int}(\ell_i))\cap (\bigcup_{i=1}^tD_i)=\emptyset,$
\item[(5)] 
for each $i=1,2,\cdots ,t-1,$ 
the arc $\ell_i$ connects a point on $\partial D_i$ and a point on $\partial D_{i+1}$, 
and 
\item[(6)] 
the arc $\ell_t$ connects a point on $\partial D_{t}$ and the point $p$.
\end{enumerate}
Let $D^+$ be a regular neighbourhood of 
$\bigcup_{i=1}^t(D_i\cup \ell_i)$ 
in $E$ and 
$D^-=Cl(E-D^+)$.
Then $D^-$ and $D^+$ are disks 
with $E=D^-\cup D^+$.
Since
$E\cap {\Bbb{BW}}(\bigcup_{i=k+1}^\infty \Gamma_i)
\subset\bigcup_{i=1}^tD_i\subset D^+$
by (1), 
we have 
\begin{enumerate}
\item[] 
$D^-\cap {\Bbb{BW}}(\bigcup_{i=k+1}^\infty \Gamma_i)=\emptyset$ 
(see Fig.~\ref{FigureFromCdiskToSeparatingSystem}(b)).
\end{enumerate}
Since 
$w_{s+1},w_{s+2},\dots ,w_{t}
\in W-\bigcup_{i=1}^s D_i
\subset E-\Gamma_k$, 
we have 
$w_{s+1}, w_{s+2}, \dots, w_{t} 
\not \in \Gamma_k$.
Thus 
\begin{enumerate}
\item[(7)] none of $D_{s+1},D_{s+2},\cdots ,D_t$ intersect 
$\bigcup_{i=1}^k \Gamma_i$.
\end{enumerate}
Since none of the movable disks at label $k$ nor 
$D_{s+1},D_{s+2},\cdots ,D_t$ intersect 
${\Bbb{BW}}(\bigcup_{i=1}^k \Gamma_i)$,
we have 
\begin{enumerate}
\item[]
$D^+\cap{\Bbb{BW}}(\bigcup_{i=1}^k \Gamma_i)=\emptyset$.
\end{enumerate}
Let $L=D^+\cap \partial E$. 
Then $L$ is a regular neighbourhood of $p$ in $\partial E$.
Since $p\in\partial E-\Gamma$, we have 
$L\cap \Gamma=\emptyset$.
Let $J=D^-\cap D^+$. 
Then $L=Cl(\partial D^+-J)$.
Thus $Cl(\partial D^+-J)\cap \Gamma=L\cap\Gamma=\emptyset$.
Since for each $i=1,2,\cdots ,t$, 
$\partial D_i\cap \Gamma_k=\emptyset$ 
by (2), (7) and Condition (ii) of a movable disk,
we have 
$J\cap \Gamma_k=\emptyset$.
Therefore the triplet $(D^-,D^+,J)$ is a desired separating system at label $k$ for $\Gamma$.
\end{proof}

\begin{figure}[h]
\centerline{\includegraphics{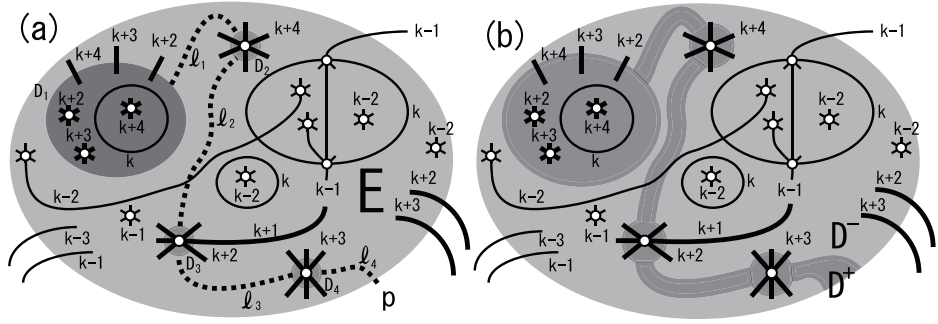}}
\caption{\label{FigureFromCdiskToSeparatingSystem}(a) The dark disk is a movable disk.}
\end{figure}

\section{$\omega_k$-minimal charts}
\label{s:OmegaMinimal}

Let $\Gamma$ be a chart.
Let $\Gamma^*$ be a chart in $Fix(\Gamma_k;\Gamma)$. 
We denote $\Gamma^*\overset{k}{\approx}\Gamma$
provided that
\begin{enumerate}
\item[] 
for any complementary domain $U$ of $\Gamma_k$, 
the domain $U$ contains mutually disjoint 
movable disks 
$D_1,D_2,\cdots ,D_s$ 
at label $k$ with respect to $\Gamma^*$ 
with 
$U\cap(\Gamma_k^*-\Gamma_k) \subset \bigcup_{i=1}^sD_i$.
\end{enumerate}
The movable disks $D_1,D_2,\cdots ,D_s$ 
are called 
{\it basic movable disks at label $k$} 
with respect to $U$ and $\Gamma^*$.
Let
\begin{enumerate}
\item[]
$\Omega(\Gamma_k;\Gamma)=
\{~\Gamma^*~|~\Gamma^*\overset{k}{\approx}\Gamma~\}$.
\end{enumerate}
Since $\Gamma\in\Omega(\Gamma_k;\Gamma)$, 
we have $\Omega(\Gamma_k;\Gamma)\neq\emptyset$.

Let $\Gamma$ be a chart,
$k$ a label of $\Gamma$,
and $\Gamma'\in \Omega(\Gamma_k;\Gamma)$.
Let
$U$ be a complementary domain of $\Gamma_k$ 
, and 
$D$ a movable disk at label $k$ in $U$ 
with respect to $\Gamma'$.
Let $F=Cl(U)$.
Let $p\in\partial F-\bigcup_{i\not=k}\Gamma_i'$ and $q\in\partial D$.
Suppose that {\it there exists  
a simple arc $\alpha$ in $F$ 
connecting the two points $p$ and $q$
with
$\alpha\cap \Gamma'=p$ and 
$\alpha\cap D=q$} 
(see Fig.~\ref{ShiftCdisk1}(a)).
Then
we can shift the movable disk $D$ at label $k$ 
to 
another complementary domain of $\Gamma_k$ 
as follows (see Fig.~\ref{ShiftCdisk1}): 
Let $N_1$ be a regular neighbourhood of 
$D\cup\alpha$ in $F$. 
Let $N_2$ be 
a regular neighbourhood of 
$N_1$ in $F$ and
$N_3$ a regular neighbourhood of 
$N_2$ in $F$.
For each $i=1,2,3$ let
$\beta_i=N_i\cap\partial F$ and 
$\gamma_i=Cl(\partial N_i-\beta_i)$ 
(see Fig.~\ref{ShiftCdisk1}(b)).
Let $\Gamma''$ be a chart with 
\begin{enumerate}
\item[]
$
\Gamma''_j=
\left\{
\begin{array}{ll}
(\Gamma'_k-\beta_3)\cup(\gamma_3\cup \partial N_1)&\hspace{4mm} \text{if $j=k$,}\\ 
\Gamma'_j &\hspace{4mm} \text{otherwise.}
\end{array}
\right. 
$
\end{enumerate}
Then 
$\Gamma'$ 
is C-move equivalent to $\Gamma''$ 
by 
modifying $\partial F$ by 
C-I-R2 moves along $\gamma_2$
and a C-I-M2 move
(see Fig.~\ref{ShiftCdisk1}(c)).
Let $N$ be a regular neighbourhood of $N_1$.
Then $N$ is a movable disk at label $k$ 
with respect to $\Gamma''$.
In a regular neighbourhood of $N_3$ 
we can modify the arc $\gamma_3$ to 
the arc $\beta_3$ 
by an ambient isotopy 
keeping $\partial \gamma_3$ fixed. 
Let $\Gamma^*$ be 
the resulting chart and
$D^*$ the disk 
modified from the movable disk $N$
(see Fig.~\ref{ShiftCdisk1}(d)).
Then $\Gamma^*$ is in $Fix(\Gamma_k;\Gamma)$ and
$D^*$ is a movable disk at label $k$ 
with respect to $\Gamma^*$.
Now $\Gamma^*_k$ is the union of 
$\Gamma'_k$ and a ring in the movable disk $D^*$.
Thus 
$\Gamma^*\overset{k}{\approx}\Gamma$.
We say that
$\Gamma^*$ is {\it obtained from $\Gamma'$ 
by shifting the movable disk $D$ 
to the outside of $Cl(U)$ 
along the arc $\alpha$} 
and that 
$D^*$ is a 
{\it movable disk induced from 
the movable disk $D$ 
circled by a ring of label $k$}. 

\begin{figure}[h]
\centerline{\includegraphics{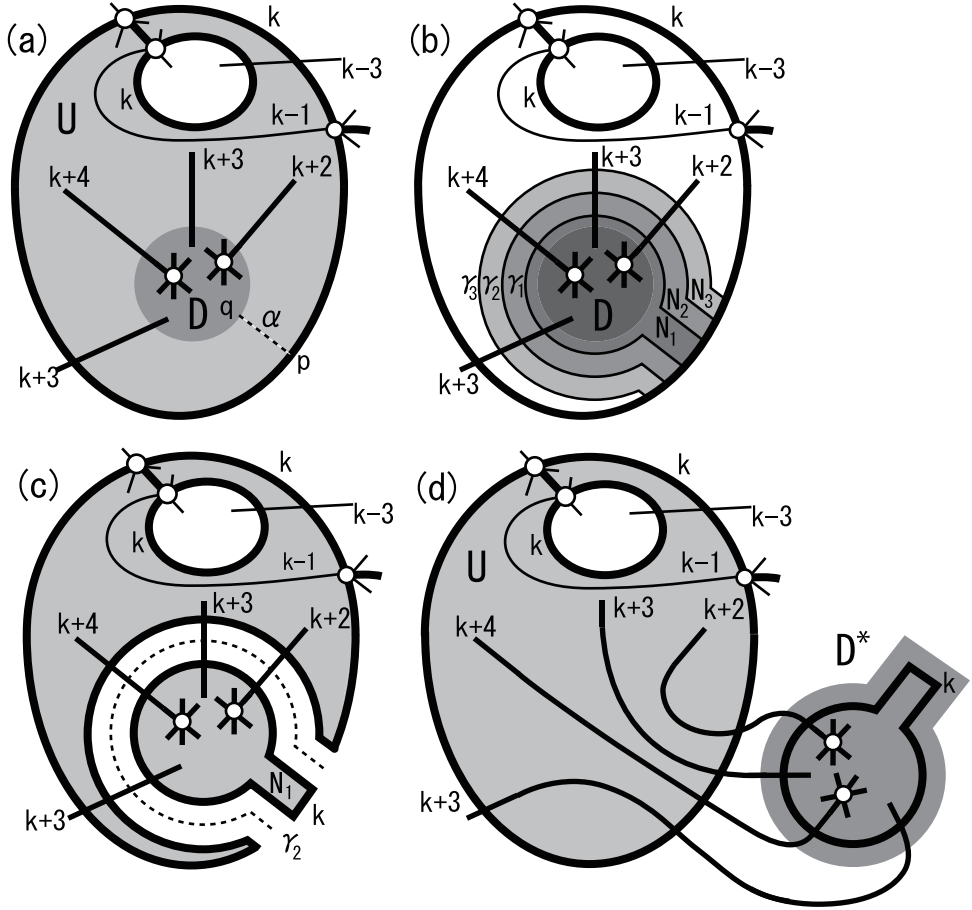}}

\caption{\label{ShiftCdisk1} Shifting a movable disk.}
\end{figure}


Let $\Gamma$ be a chart and 
$k$ a label of $\Gamma$. 
Let $T$ be a maximal tree of the dual graph of 
$\Gamma_k$. 
Namely 
\begin{enumerate}
\item[(i)] 
each vertex $v$ of the tree $T$ 
corresponds to 
a complementary domain $U_v$ 
of $\Gamma_k$, and
\item[(ii)] 
each edge $e$ of the tree $T$ with 
$\partial e=\{ v_1,v_2 \}$ 
corresponds to an edge 
$e_\Gamma$ of $\Gamma_k$ 
with $e_\Gamma\subset 
Cl(U_{v_1})\cap Cl(U_{v_2})$.
\end{enumerate}
The tree $T$ is called a {\it dual tree 
with respect to $\Gamma_k$}.

Let $\Gamma$ be a chart and 
$k$ a label of $\Gamma$. 
Let $T$ be a dual tree with respect to 
$\Gamma_k$. 
Let $v_0$ be the vertex of $T$ 
which corresponds to  
the unbounded complementary domain 
of $\Gamma_k$.
Let $V(T)$ be the set of all vertices of $T$. 
For each vertex $v\in V(T)$,
let $T(v,v_0)$ be the path in $T$ connecting 
$v$ and $v_0$. 
Let
\begin{enumerate}
\item[] 
$length(v;T)=$ the number of edges in $T(v,v_0)$.
\end{enumerate}
We have $length(v_0;T)=0$. 
For each chart 
$\Gamma^*\in \Omega(\Gamma_k;\Gamma)$ 
and 
$v\in V(T)$, 
let
\begin{enumerate}
\item[]
$
weight_k(v;\Gamma^*)=
\left\{
\begin{array}{ll}
0\hspace{10mm}&
\text{if $U_v\cap {\Bbb{BW}}(\bigcup_{i=k+1}^\infty\Gamma_i^*)=\emptyset$},\\
1& \text{otherwise,}
\end{array}
\right.
$
\item[] 
$\omega_k(\Gamma^*,T)=
{\sum_{v\in V(T)}weight_k(v;\Gamma^*)\times length(v;T)}$.
\end{enumerate}

A chart $\Gamma'\in \Omega(\Gamma_k;\Gamma)$ 
is {\it $\omega_k$-minimal}
if 
there exists a dual tree $T$ 
with respect to 
$\Gamma_k$ such that 
$\omega_k(\Gamma',T)=\min\{~\omega_k(\Gamma^*,T)~|~\Gamma^*\in \Omega(\Gamma_k;\Gamma)~\}$.

\begin{lemma}
\label{LemmaOmega_kMinimal}
Let $\Gamma$ be a chart 
which is zero at label $k$. 
If a chart $\Gamma'
\in \Omega(\Gamma_k;\Gamma)$ is
$\omega_k$-minimal, then 
there exists a dual tree $T$ 
with respect to $\Gamma_k$ 
with $\omega_k(\Gamma',T)=0$.
\end{lemma}

\begin{proof} 
Suppose that 
for a dual tree $T$ with respect to 
$\Gamma_k$, we have
\begin{enumerate}
\item[]
$\omega_k(\Gamma',T)
=\min\{~\omega_k(\Gamma^*,T)~|~\Gamma^*\in \Omega(\Gamma_k;\Gamma)~\}$.
\end{enumerate}
We shall show $\omega_k(\Gamma',T)=0$ 
by contradiction.
Suppose that
$\omega_k(\Gamma',T)>0$.
Let $V(T)$ be the set of all vertices of $T$ and
\begin{enumerate}
\item[]
${\Bbb P}=\{~v\in V(T)~|~weight_k(v;\Gamma')\times length(v;T)>0~\}$. 
\end{enumerate}
Then $\omega_k(\Gamma',T)>0$ implies 
${\Bbb P}\neq\emptyset$.
Let $v_1$ be a vertex in $\Bbb P$ such that
\begin{enumerate}
\item[(1)]
$length(v_1;T)=\max\{~length(v;T)~|~v\in {\Bbb P}~\}$.
\end{enumerate}
Then $v_1\in \Bbb P$ implies
\begin{enumerate}
\item[(2)]
$0<weight_k(v_1;\Gamma')\times length(v_1;T)=length(v_1;T)$.
\end{enumerate}

Let 
$V_1$ be a complementary domain 
of $\Gamma_k$ 
corresponding to $v_1$. 
Since $\Gamma'\in\Omega(\Gamma_k;\Gamma)$, 
there exist 
basic movable disks 
$D_1,D_2,\cdots ,D_s$ 
at label $k$ 
with respect to $V_1$ and $\Gamma'$ 
(see Fig.~\ref{FigureOmegakMinimal}(a)) 
such that
\begin{enumerate}
\item[(3)]
$V_1\cap(\Gamma'_k-\Gamma_k)\subset \bigcup_{i=1}^sD_i$.
\end{enumerate}
Let $\delta$ be 
a connected component of $\partial V_1$ 
such that 
a bounded complementary domain of $\delta$ 
contains $V_1$
(see Fig.~\ref{FigureOmegakMinimal}(b)).
Let $N(\delta)$ be 
a regular neighbourhood of $\delta$ 
and 
$W=Cl(V_1-N(\delta))$
(see Fig.~\ref{FigureOmegakMinimal}(c)). 
Let $\ell$ be 
a connected component of $\partial W$ 
such that
$W$ is contained in 
the disk $E$ 
bounded by $\ell$
(see Fig.~\ref{FigureOmegakMinimal}(c) and (d)). 
Then we have 
\begin{enumerate}
\item[]
$W\supset \partial E$.
\end{enumerate}
\begin{figure}[h]
\centerline{\includegraphics{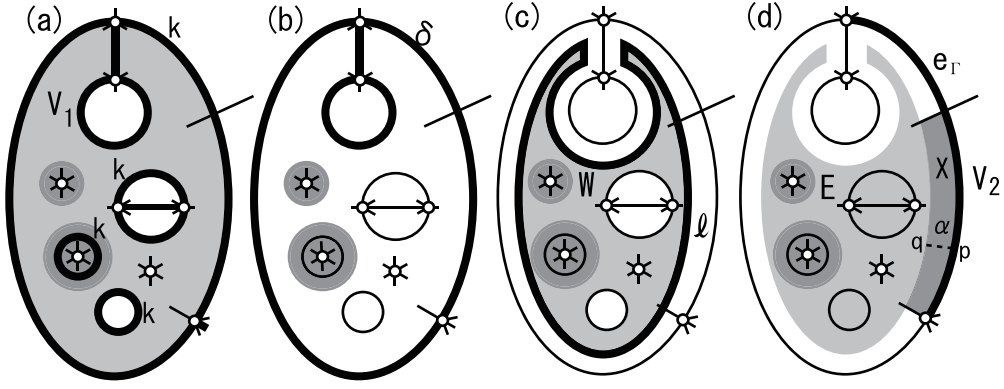}}
\caption{\label{FigureOmegakMinimal} Dark disks are movable disks at label $k$. 
(a)~The thicken curves are of label $k$. 
(b)~The thicken curve is the set $\delta$. 
(c)~The thicken curve is the simple closed curve $\ell$. 
(d)~The thicken curve is the edge $e_\Gamma$. 
The set $X$ is a disk with ${\rm Int}X\cap {\rm Int}E=\emptyset$.}
\end{figure}


{\bf Claim.} The disk $E$ is a c-disk.

{\it Proof of Claim.}
Since $\Gamma'\overset{k}{\approx}\Gamma$, 
the chart $\Gamma'$ is zero at label $k$. 
Hence 
$\delta\subset \Gamma_k\subset \Gamma'_k$ 
implies that
$\delta\cap \Gamma'_{k+1}=\emptyset$. Thus
\begin{enumerate}
\item[(4)]
$\partial E\cap(\Gamma'_k\cup\Gamma'_{k+1})=\emptyset$.
\end{enumerate}
Now 
$\delta\cap (\bigcup_{i=1}^sD_i)=\emptyset$
implies 
$N(\delta)\cap (\bigcup_{i=1}^sD_i)=\emptyset$. 
Thus 
$V_1\supset \bigcup_{i=1}^sD_i$ 
implies 
\begin{enumerate}
\item[]
$W=Cl(V_1-N(\delta))\supset 
\bigcup_{i=1}^sD_i$.
\end{enumerate}
Let $\Bbb S$ be the set of 
all connected components of ${E-(\Gamma_k-\bigcup_{i=1}^sD_i)}$ 
different from $W$. 
Then we have
\begin{enumerate}
\item[(5)]
$E=W\cup (\bigcup\{~Cl(U)~|~U\in{\Bbb S}~\})$.
\end{enumerate}

For each domain $U\in \Bbb S$, 
let $v_U$ be the vertex of $T$ corresponding to $U$. 
Since $U$ is surrounded by $V_1$, 
we have $v_1\in T(v_U,v_0)$
where
$v_0$ is the vertex of $T$ which 
corresponds to 
the unbounded complementary domain 
of $\Gamma_k$. Hence
\begin{enumerate}
\item[] 
$length(v_U;T)\geq  length(v_1;T)+1$ and
\item[] 
$weight_k(v_U;\Gamma')=0$ 
\end{enumerate}
by the property (1) of the vertex $v_1$. 
This means
\begin{enumerate}
\item[] 
$U\cap 
{\Bbb {BW}}(\bigcup_{i=k+1}^\infty\Gamma'_i)
=\emptyset$.
\end{enumerate}
Since $\Gamma'\overset{k}{\approx}\Gamma$, 
the chart $\Gamma'$ is zero at label $k$. 
Thus
$\Gamma'_k\cap \Gamma'_{k+1}=\emptyset$. 
Hence $\partial U\subset \Gamma_k\subset \Gamma'_k$ 
implies  
\begin{enumerate}
\item[] 
$Cl(U)\cap 
{\Bbb {BW}}(\bigcup_{i=k+1}^\infty\Gamma'_i)
=\emptyset$.
\end{enumerate}
Thus (5)
implies that 
\begin{enumerate}
\item[] 
$E\cap 
{\Bbb {BW}}(\bigcup_{i=k+1}^\infty\Gamma'_i)
\subset W$.
\end{enumerate}
Therefore 
the disk $E$ is 
a c-disk at label $k$ 
with respect to $\Gamma'$.
Hence Claim holds.

Let $T(v_1,v_0)$ be the path in the tree $T$ 
connecting $v_1$ and $v_0$. 
Let $e$ be the edge of $T(v_1,v_0)$ 
with $e\ni v_1$. 
Let $v_2$ be the vertex of $e$ different from $v_1$. 
Now we have 
\begin{enumerate}
\item[(6)]
$length(v_2;T)=length(v_1;T)-1$.
\end{enumerate}
Let $e_\Gamma$ be the edge of $\Gamma_k$ 
corresponding to the edge $e$ 
which connects the two vertices 
$v_1$ and $v_2$. 
Let $V_2$ be the complementary domain of $\Gamma_k$ 
corresponding to $v_2$. 
Then we have 
$e_\Gamma\subset Cl(V_1)\cap Cl(V_2)$. 
Let $p$ be a point in 
${\rm Int}(e_\Gamma)-\bigcup_{i\neq k}\Gamma'_i$ 
and  
$X$ the closure of 
a connected component of
$(V_1-E)-\Gamma' $ with $X\ni p$. 
Let $q$ be a point in $\partial E\cap X$ and 
$\alpha$ a proper simple arc in $X$ 
with $\alpha\cap\Gamma'=p$ 
and 
$\alpha\cap E=q$ 
(see Fig.~\ref{FigureOmegakMinimal}(d)). 

By Lemma~\ref{LemmaFromC-diskToSeparatingSystem} 
there exists a separating system $(D^-,D^+,J)$ 
at label $k$ 
for $\Gamma'$ with 
$E=D^-\cup D^+$ 
and $q\in D^+\cap \partial E$.

Now 
(4) implies 
\begin{enumerate}
\item[]
$\partial (D^-\cup D^+)\cap \Gamma'_{k+1}=
\partial E\cap \Gamma'_{k+1}=\emptyset$. 
\end{enumerate}
Thus 
by Lemma~\ref{LemmaMovableDisk} 
there exists a chart $\Gamma''\in Fix(\Gamma_k;\Gamma)$ obtained from
$\Gamma'$ by applying surgeries 
along subarcs of $J$ such that
$D^+$ is a movable disk at label $k$ 
with respect to $\Gamma''$.

Let $\Gamma^*$ be a chart 
obtained from $\Gamma''$ 
by shifting the movable disk $D^+$ 
to the outside of $Cl(V_1)$ 
along the arc $\alpha$ 
and 
let $D^{*}$ be a movable disk induced from 
the movable disk $D^+$ 
circled by a ring of label $k$.
Then $D^*$ is a new movable disk in $V_2$ 
disjoint from the old movable disks in $V_2$.
Hence we have 
$\Gamma^*\in\Omega(\Gamma_k;\Gamma)$. 
Now\\
$
weight_k(v_1;\Gamma')\times length(v_1;T)=length(v_1;T)>0$ by (2),\\ 
$weight_k(v_1;\Gamma^*)\times length(v_1;T)=0,$\\ 
$weight_k(v_2;\Gamma^*)\times length(v_2;T)=length(v_2;T)
=length(v_1;T)-1$ by (6), and\\
$weight_k(v;\Gamma^*)\times length(v;T)=
weight_k(v;\Gamma')\times length(v;T)~~
(v\neq v_1,v_2)$\\
imply that
\begin{enumerate}
\item[] 
$\omega_k(\Gamma^*,T)\leq \omega_k(\Gamma',T)-1$.
\end{enumerate}
This contradicts the fact that $\Gamma'$ is 
$\omega_k$-minimal.
Therefore $\omega_k(\Gamma',T)=0$.
\end{proof}

\section{Proof of Main Theorem}
\label{s:ProofMainTheorem}


{\it Proof of Theorem~\ref{MainTheorem}.}
Suppose that a chart $\Gamma$ is zero at label $k$. 
Then there exists a chart $\Gamma'\in \Omega(\Gamma_k;\Gamma)$ 
such that 
$\Gamma'$ is $\omega_k$-minimal. 
Since $\Gamma$ is zero at label $k$, 
the chart $\Gamma'$ is zero at label $k$, too. Thus 
\begin{enumerate}
\item[(1)] $\Gamma_k'\cap\Gamma_{k+1}'=\emptyset$.
\end{enumerate}
Let $E$ be a disk containing the chart $\Gamma'$ 
in its inside,
i.e. Int$E\supset \Gamma'$.
By Lemma~\ref{LemmaOmega_kMinimal} 
there exists a dual tree $T$ 
with respect to 
$\Gamma_k$ with $\omega_k(\Gamma',T)=0$. 
Namely 
${\Bbb {BW}}(\bigcup_{i=k+1}^\infty\Gamma'_i)$ 
is contained in  
the unbounded complementary domain $U$ 
of $\Gamma_k$. 
Now 
$U\cap \Gamma_k=\emptyset$.
Since 
$\Gamma'\in\Omega(\Gamma_k;\Gamma)$,
the set $U\cap (\Gamma'_k-\Gamma_k)$ is 
contained in the union of
basic movable disks of label $k$ with respect to $U$ 
and $\Gamma'$.
Furthermore, since $E\cap U$ is connected, 
by the same way as the one in 
Lemma~\ref{LemmaFromC-diskToSeparatingSystem} 
there exists a disk $D^+$ in $E$ such that
\begin{enumerate}
\item[(2)] 
$D^+$ is in general position 
with respect to $\Gamma'$,
\item[(3)] 
$D^+\cap \partial E$ is an arc,
\item[(4)] 
$\partial D^+\cap \Gamma'_k=\emptyset$, and
\item[(5)] 
${\Bbb {BW}}(\Gamma')\cap D^+=
{\Bbb {BW}}(\bigcup_{i=k+1}^\infty\Gamma'_i)$.
\end{enumerate}
Let 
$D^-=Cl(E-D^+)$ and $J=D^-\cap D^+$.
Then 
we have 
\begin{enumerate}
\item[(6)] $D^-\cap {\Bbb {BW}}(\bigcup_{i=k+1}^\infty\Gamma'_i)=\emptyset$ 
and 
$D^+\cap {\Bbb {BW}}(\bigcup_{i=1}^k\Gamma'_i)=\emptyset$ by (1) and (5), 
\item[(7)] $Cl(\partial D^+-J)\cap\Gamma'\subset \partial E\cap\Gamma'=\emptyset$
by Int$E\supset\Gamma'$,
\item[(8)] $J\cap\Gamma_k'\subset \partial D^+\cap\Gamma_k'=\emptyset$ by (4).
\end{enumerate}
Hence the triplet 
$(D^-,D^+,J)$ is a separating system at label $k$ for $\Gamma'$.
By Lemma~\ref{Lemma(D,J)-arcFree} 
there exists a chart 
$\Gamma''\in Fix(\Gamma_k;\Gamma)$ 
obtained from $\Gamma'$ 
by applying surgeries along 
subarcs of $J$ such that 
\begin{enumerate}
\item[(9)] $\Gamma''_k=\Gamma'_k$,
\item[(10)] $\Gamma''$ does not possess any 
$(D^+,J)$-arcs of label less than $k$
nor 
$(D^-,J)$-arcs of label 
greater than $k$, 
\item[(11)] $D^-\cap {\Bbb {BW}}(\bigcup_{i=k+1}^\infty\Gamma''_i)=\emptyset$ and $D^+\cap {\Bbb {BW}}(\bigcup_{i=1}^k\Gamma''_i)=\emptyset$ by (6). 
\end{enumerate}
Now (11) and $\partial (D^-\cup D^+)\cap \Gamma''=
\partial E\cap \Gamma''=\emptyset$ 
imply that
\begin{enumerate}
\item[(12)]
for each label $i$ with $i< k$ 
if a connected component of $\Gamma''_i\cap D^+$ 
intersects $J$, 
then the component is a 
$(D^+,J)$-arc, 
\item[(13)]
for each label $i$ with $k<i$ 
if a connected component of 
$\Gamma''_i\cap D^-$ 
intersects $J$, 
then the component is a 
$(D^-,J)$-arc.
\end{enumerate}
Since $J\cap\Gamma''_k=J\cap\Gamma_k'=\emptyset$ by (8) and (9),
we have 
$J\cap\Gamma''=\emptyset$ 
by (10), (12) and (13).
Thus 
$(\partial D^+-J)\subset \partial E$
implies 
\begin{enumerate}
\item[]
$\partial D^+\cap \Gamma''=\emptyset$.
\end{enumerate}
Let 
$\Gamma^*=D^-\cap \Gamma''$ and 
$\Gamma^{**}=D^+\cap \Gamma''$. 
Then 
$\Gamma''$ is the product of $\Gamma^*$ and $\Gamma^{**}$.
Now 
$\Gamma''\in Fix(\Gamma_k;\Gamma)$ 
implies that  $\Gamma''$ is zero at label $k$.
Thus there exist two labels $i$ and $j$ 
with $i\leq k<j$, 
$w(\Gamma''_i)\neq 0$ and 
$w(\Gamma''_j)\neq 0$.
Namely 
$w(\Gamma^{*})>0$ and $w(\Gamma^{**})>0$.
Now 
(11)
implies
\begin{enumerate}
\item[]
$w(\Gamma^{*}_i)=0$ for all label $i$ 
with $k<i$, and
\item[]
$w(\Gamma^{**}_i)=0$ for all label $i$ 
with $i\leq k$.
\end{enumerate}
Therefore 
$\Gamma''$ is separable at label $k$.
It is clear that $\Gamma''$ and $\Gamma$ are same C-type.

Conversely if $\Gamma$ is separable at label $k$,
then 
the chart $\Gamma$ is clearly zero at label $k$.
Thus we have done.
\hfill $\square$

\begin{lemma}
\label{N1NP}
{\rm (\cite[Lemma 6.1, Proposition 6.6 and Proposition 6.11]{ChartAppII})}
Let $\Gamma$ be a minimal chart of type $(n_1,n_2,\cdots,n_p)$.
Then we have the following:
\begin{enumerate}
\item[{\rm (a)}] $n_1>1$ and $n_p>1$.
\item[{\rm (b)}] If $n_1=2$ $($resp. $n_p=2)$, then $n_2>1$ $($resp. $n_{p-1}>1)$.
\item[{\rm (c)}] If $n_1=3$ $($resp. $n_p=3)$, then $n_2>1$ $($resp. $n_{p-1}>1)$.
\end{enumerate}
\end{lemma}

{\it Proof of Corollary~\ref{MinimalChart8Zero}.}
Let $\Gamma$ be a minimal chart with $w(\Gamma)=8$.
Suppose that $\Gamma$ is zero at label $k$. 
Then by Theorem~\ref{MainTheorem} 
there exists a chart $\Gamma'$ with the same C-type of $\Gamma$
which is separable at label $k$.
Here $w(\Gamma')=w(\Gamma)$.
Hence there exist
two subcharts $\Gamma^*$, $\Gamma^{**}$ 
such that 
\begin{enumerate}
\item[(1)] $\Gamma'$ is the product of $\Gamma^*$ and $\Gamma^{**}$,
\item[(2)] $w(\Gamma^*)\not=0$ and 
$w(\Gamma^{**})\not=0$.
\end{enumerate}
Then we have
\begin{enumerate}
\item[(3)]
$w(\Gamma^*)+w(\Gamma^{**})=w(\Gamma')=8$.
\end{enumerate}
If either $\Gamma^*$ or $\Gamma^{**}$ is not minimal,
then $\Gamma'$ is C-move equivalent to a chart $\Gamma''$ with $w(\Gamma'')<w(\Gamma')=w(\Gamma)$.
Namely the chart $\Gamma$ is 
C-move equivalent to $\Gamma''$.
This contradicts the fact that $\Gamma$ is minimal.
Hence the two charts $\Gamma^*$ and $\Gamma^{**}$ are minimal.
Since there does not exist a minimal chart 
with at most three white vertices,
we have $w(\Gamma^*)\ge4$ and $w(\Gamma^{**})\ge4$.
Thus by (3),
we have $w(\Gamma^*)=4$ and $w(\Gamma^{**})=4$.
By Lemma~\ref{N1NP}(a),
each of $\Gamma^*$ and $\Gamma^{**}$
is of type $(4)$ or $(2,2)$.
Therefore we complete the proof of Corollary~\ref{MinimalChart8Zero}.
\hfill $\square$\\

{\it Proof of Corollary~\ref{MinimalChart8}.}
 Let $\Gamma$ be a minimal $n$-chart with $w(\Gamma)=8$ of type $(n_1,n_2,\cdots,n_p)$
 such that $\Gamma$ is not zero at any label.
 Then 
\begin{enumerate}
\item[(1)] $n_1+n_2+\cdots+n_p=w(\Gamma)=8$,
\item[(2)] $n_i\ge1$ for each $i$ ($1\le i\le p$).
\end{enumerate}
By Lemma~\ref{N1NP}(a),
we have $n_1\ge2$ and $n_p\ge2$.
If necessary
we change all the edges of label $k$ to 
ones of label $n-k$
for each $k=1,2,\cdots,n-1$
simultaneously,
we can assume that
\begin{enumerate}
\item[(3)] $n_1\ge n_p\ge2$.
\end{enumerate}
There are four cases:
(i) $p=1$ or $2$,
(ii) $p=3$,
(iii) $p=4$,
(iv) $p\ge5$.

{\bf Case (i).}
If $p=1$,
then $\Gamma$ is of type $(8)$.
If $p=2$,
then by (1) and (3)
we have that the chart $\Gamma$ is of type 
$(6,2)$, $(5,3)$ or $(4,4)$.

{\bf Case (ii).}
Suppose $p=3$.
If $n_3=2$ or $3$,
then $n_2=n_{p-1}\ge2$ by Lemma~\ref{N1NP}(b) and (c).
Thus by (1) and (3)
we have that the chart $\Gamma$ is of type 
$(4,2,2)$, $(3,3,2)$, $(2,4,2)$ or $(3,2,3)$.

If $n_3\ge4$,
then by (2) and (3)
we have that $w(\Gamma)=n_1+n_2+n_3\ge 4+1+4=9$.
This contradicts the fact $w(\Gamma)=8$.

{\bf Case (iii).}
Suppose $p=4$.
If $n_4=2$ or $3$,
then $n_3=n_{p-1}\ge2$ by Lemma~\ref{N1NP}(b) and (c).
Thus by (1), (2) and (3)
we have that the chart $\Gamma$ is of type 
$(2,2,2,2)$, $(2,1,3,2)$ or $(3,1,2,2)$.
However
if $n_1=2$ or $3$, then $n_2\ge2$ by Lemma~\ref{N1NP}(b) and (c).
Hence $\Gamma$ is of type 
$(2,2,2,2)$.

If $n_4\ge4$,
then by (2) and (3)
we have that $w(\Gamma)=n_1+n_2+n_3+n_4\ge 4+1+1+4=10$.
This contradicts the fact $w(\Gamma)=8$.

{\bf Case (iv).}
Suppose $p\ge5$. 
There are two cases:
(iv-1) $n_p=2$ or $3$,
(iv-2) $n_p\ge4$.

{\bf Cases (iv-1).}
We have $n_{p-1}\ge2$ by Lemma~\ref{N1NP}(b) and (c).
Thus 
by (2), (3)
we have that
$w(\Gamma)=n_1+n_2+\cdots+n_{p-1}+n_p
\ge n_1+n_2+n_3+n_{p-1}+n_p \ge
n_1+n_2+1+2+2=n_1+n_2+5$.
Hence
$$w(\Gamma)\ge n_1+n_2+5.$$

If $n_1=2$,
then $n_{2}\ge2$ by Lemma~\ref{N1NP}(b).
Thus $w(\Gamma)\ge 2+2+5=9$.
This contradicts the fact $w(\Gamma)=9$.

If $n_1\ge3$,
then by (2)
we have $w(\Gamma)\ge 3+1+5=9$.
This contradicts the fact $w(\Gamma)=9$.

{\bf Cases (iv-2).}
Since $n_p\ge4$,
by (2) and (3)
we have that
$w(\Gamma)=n_1+n_2+\cdots+n_{p-1}+n_p
\ge n_1+n_2+n_3+n_{p-1}+n_p \ge
4+1+1+1+4=11$.
This contradicts the fact $w(\Gamma)=9$.
\hfill $\square$\\


A chart $\Gamma$ belongs to {\it the first class} provided that
\begin{enumerate}
\item[(i)] $w(\Gamma)$ is odd, and
\item[(ii)] there does not exist a minimal chart $\Gamma'$ such that 
$w(\Gamma')$ is odd and less than $w(\Gamma)$.
\end{enumerate}


\begin{corollary}
\label{Corollary2}
If a minimal chart belongs to the first class,
then the chart is not zero at any label.
Namely 
if the type of the chart is $(m;n_1,n_2,\cdots ,n_p)$, 
then for each $i=1,2,\cdots,p$, we have $n_i\neq 0$.
\end{corollary}

\begin{proof}
Let 
$\Gamma$ be a minimal chart belonging to the first class.
Suppose that the chart is zero at label $k$.
Then by Theorem~\ref{MainTheorem} 
there exists a chart $\Gamma'$ with the same C-type of $\Gamma$
which is separable at label $k$.
Thus there exist
subcharts $\Gamma^*$ and $\Gamma^{**}$ 
such that
\begin{enumerate}
\item[(1)] $\Gamma'$ is the product of the two charts $\Gamma^{*}$ and $\Gamma^{**}$,
\item[(2)] $w(\Gamma^{*})\neq 0$ and
$w(\Gamma^{**})\neq 0$.
\end{enumerate}
Then we have
\begin{enumerate}
\item[(3)]
$w(\Gamma^*)+w(\Gamma^{**})=w(\Gamma')$.
\end{enumerate}
Since $\Gamma'$ and $\Gamma$ are same C-type, 
we have $w(\Gamma')=w(\Gamma)$. 
Since $\Gamma$ belongs to the first class, 
we have that $w(\Gamma')(=w(\Gamma))$ is odd. 
Thus $w(\Gamma^*)$ or $w(\Gamma^{**})$ is odd
less than $w(\Gamma)$.
Since $\Gamma$ belongs to the first class, 
either $\Gamma^*$ or $\Gamma^{**}$ is not 
a minimal chart.
Hence $\Gamma'$ is C-move equivalent to a chart $\Gamma''$
with $w(\Gamma'')<w(\Gamma')=w(\Gamma)$.
Namely the chart $\Gamma$ is 
C-move equivalent to $\Gamma''$. 
This contradicts the fact that $\Gamma$ is minimal. 
Therefore $\Gamma$ is not zero at any label.
\end{proof}

{\it Proof of Corollary~\ref{Corollary1}.} 
There does not exist a minimal chart $\Gamma$ 
with $w(\Gamma)=1,2,$ or $3$.
Further there does not exist a minimal chart $\Gamma$ 
with $w(\Gamma)=5$ (\cite{ONS}).
Furthermore there does not exist a minimal chart $\Gamma$ 
with $w(\Gamma)=7$ 
(\cite{ChartAppl}, \cite{ChartAppII}, \cite{ChartAppIII}, \cite{ChartAppIV}, \cite{ChartAppV}).
Hence any chart with nine white vertices 
belongs to the first class.
Thus any minimal chart 
with nine white vertices
is not zero at any label 
by Corollary~\ref{Corollary2}.

Let $\Gamma$ be a minimal chart with $w(\Gamma)=11$.
Suppose that $\Gamma$ is zero at label $k$. 
Then by Theorem~\ref{MainTheorem} 
there exists a chart $\Gamma'$ with the same C-type of $\Gamma$
which is separable at label $k$.
Here $w(\Gamma')=w(\Gamma)$.
Hence there exist
two subcharts $\Gamma^*$, $\Gamma^{**}$ such that
\begin{enumerate}
\item[(1)] $\Gamma'$ is the product of the two charts $\Gamma^{*}$ and $\Gamma^{**}$,
\item[(2)] $w(\Gamma^{*})\neq 0$ and 
$w(\Gamma^{**})\neq 0$.
\end{enumerate}
Then we have
\begin{enumerate}
\item[(3)]
$w(\Gamma^*)+w(\Gamma^{**})=w(\Gamma')=11$.
\end{enumerate}
There are five cases:
\begin{enumerate}
\item[(i)] 
one of $w(\Gamma^*)$ and $w(\Gamma^{**})$ equals ${\bf 1}$ 
and the other equals 10,
\item[(ii)] 
one of $w(\Gamma^*)$ and $w(\Gamma^{**})$ equals ${\bf 2}$ 
and the other equals 9,
\item[(iii)] 
one of $w(\Gamma^*)$ and $w(\Gamma^{**})$ equals ${\bf 3}$ 
and the other equals 8,
\item[(iv)] 
one of $w(\Gamma^*)$ and $w(\Gamma^{**})$ equals $4$ 
and the other equals {$\bf 7$}, and
\item[(v)] 
one of $w(\Gamma^*)$ and $w(\Gamma^{**})$ equals ${\bf 5}$ 
the other equals 6.
\end{enumerate}
But in any case, 
either $\Gamma^*$ or $\Gamma^{**}$ 
is not minimal.
Hence $\Gamma'$ is C-move equivalent to a chart $\Gamma''$
with $w(\Gamma'')<w(\Gamma')=w(\Gamma)$.
Namely the chart $\Gamma$ is 
C-move equivalent to $\Gamma''$.
This contradicts the fact that $\Gamma$ is minimal. 
Therefore any minimal chart $\Gamma$ 
with $w(\Gamma)=11$ is 
not zero at any label.
\hfill $\square$\\

{\it Proof of Corollary~\ref{MinimalChart9}.}
 Let $\Gamma$ be a minimal $n$-chart with $w(\Gamma)=9$ of type $(n_1,n_2,\cdots,n_p)$.
 Then 
\begin{enumerate}
\item[(1)] $n_1+n_2+\cdots+n_p=w(\Gamma)=9$.
\end{enumerate}
By Corollary~\ref{Corollary1},
we have 
\begin{enumerate}
\item[(2)] $n_i\ge1$ for each $i$ ($1\le i\le p$).
\end{enumerate}
By Lemma~\ref{N1NP}(a),
we have $n_1\ge2$ and $n_p\ge2$.
If necessary
we change all the edges of label $k$ to 
ones of label $n-k$
for each $k=1,2,\cdots,n-1$
simultaneously,
we can assume that
\begin{enumerate}
\item[(3)] $n_1\ge n_p\ge2$, and
\item[(4)] if $n_1=n_p$ and $p\ge4$, then $n_2\ge n_{p-1}$.
\end{enumerate}

If $p=1$,
then $\Gamma$ is of type $(9)$.
If $p=2$, then 
by (1) and (3) the chart $\Gamma$ is of type 
$(7,2)$, $(6,3)$ or $(5,4)$.

Suppose $p\ge3$.
There are two cases:
(i) $n_p=2$ or $3$, (ii) $n_p\ge4$.

{\bf Case (i).}
By Lemma~\ref{N1NP}(b) and (c),
we have $n_{p-1}\ge2$.
There are four cases:
(i-1) $p=3$,
(i-2) $p=4$,
(i-3) $p=5$,
(i-4) $p\ge6$.

{\bf Case (i-1).} 
Suppose $p=3$. 
By (1), (3) and $n_2=n_{p-1}\ge2$,
the chart $\Gamma$ is of type 
$(5,2,2)$, $(4,3,2)$, $(3,4,2)$, $(2,5,2)$, 
$(4,2,3)$ or $(3,3,3)$.

{\bf Case (i-2).} 
Suppose $p=4$. 
If $n_1=2$ or $3$,
then $n_2\ge2$ by Lemma~\ref{N1NP}(b) and (c).
Thus by (1), (3), (4) and $n_3=n_{p-1}\ge2$,
the chart $\Gamma$ is of type 
$(2,3,2,2)$ or $(3,2,2,2)$.

If $n_1=4$,
by (1), (2), (3) and $n_3=n_{p-1}\ge2$,
then $\Gamma$ is of type $(4,1,2,2)$

If $n_1\ge5$,
then by (2), (3) and $n_3=n_{p-1}\ge2$,
we have $w(\Gamma)=n_1+n_2+n_3+n_4\ge5+1+2+2=10$.
This contradicts the fact $w(\Gamma)=9$.

{\bf Case (i-3).} 
Suppose $p=5$. 
If $n_1=2$,
then $n_2\ge2$ by Lemma~\ref{N1NP}(b).
Thus by (1), (2), (3) and $n_4=n_{p-1}\ge2$,
we have that the chart $\Gamma$ is of type 
$(2,2,1,2,2)$.

If $n_1=3$,
then $n_2\ge2$ by Lemma~\ref{N1NP}(c).
Thus by (2), (3) and $n_4=n_{p-1}\ge2$
we have 
$w(\Gamma)=n_1+n_2+n_3+n_4+n_5\ge3+2+1+2+2=10$.
This contradicts the fact $w(\Gamma)=9$.

If $n_1\ge4$,
then by (2), (3) and $n_4=n_{p-1}\ge2$
we have
$w(\Gamma)=n_1+n_2+n_3+n_4+n_5\ge4+1+1+2+2=10$.
This contradicts the fact $w(\Gamma)=9$.

{\bf Case (i-4).} 
Suppose $p\ge6$. 
If $n_1=2$ or $3$,
then $n_2\ge2$ by Lemma~\ref{N1NP}(b) and (c).
Thus 
by (2), (3) and $n_{p-1}\ge2$,
we have
$w(\Gamma)=n_1+n_2+\cdots+n_{p-1}+n_p
\ge n_1+n_2+n_3+n_4+n_{p-1}+n_p \ge
2+2+1+1+2+2=10$.
This contradicts the fact $w(\Gamma)=9$.

If $n_1\ge4$,
then 
by (2), (3) and $n_{p-1}\ge2$
we have
$w(\Gamma)=n_1+n_2+\cdots+n_{p-1}+n_p
\ge n_1+n_2+n_3+n_4+n_{p-1}+n_p \ge
4+1+1+1+2+2=11$.
This contradicts the fact $w(\Gamma)=9$.

{\bf Case (ii).}
If $p=3$,
then by (1), (2) and (3)
we have $\Gamma$ is of type $(4,1,4)$.

If $p\ge4$,
then
by (2) and (3)
we have
$w(\Gamma)=n_1+n_2+\cdots+n_{p-1}+n_p
\ge n_1+n_2+n_3+n_p \ge
4+1+1+4=10$.
This contradicts the fact $w(\Gamma)=9$.
\hfill $\square$




\vspace{5mm}

\begin{minipage}{65mm}
{Teruo NAGASE
\\
{\small Tokai University \\
4-1-1 Kitakaname, Hiratuka \\
Kanagawa, 259-1292 Japan\\
\\
nagase@keyaki.cc.u-tokai.ac.jp
}}
\end{minipage}
\begin{minipage}{65mm}
{Akiko SHIMA 
\\
{\small Department of Mathematics, 
\\
Tokai University
\\
4-1-1 Kitakaname, Hiratuka \\
Kanagawa, 259-1292 Japan\\
shima@keyaki.cc.u-tokai.ac.jp
}}
\end{minipage}

\end{document}